\documentclass[12pt,a4paper,reqno]{amsart}
\usepackage{amssymb,mathrsfs}
\newtheorem{thm}{Theorem}[section]

\newtheorem{lem}[thm]{Lemma}

\theoremstyle{definition}
\newtheorem{defn}[thm]{Definition}

\newcommand{\cl}[1]{\ensuremath{\overline{{#1}}}}
\newcommand{\ep}{\varepsilon}
\newcommand{\map}[3]{\ensuremath{{#1}:{#2}\to{#3}}}
\newcommand{\n}[1]{\ensuremath{\left\|{#1}\right\|}}
\newcommand{\ndot}{\ensuremath{\left\|\cdot\right\|}}
\newcommand{\pn}[2]{\ensuremath{\left\|{#1}\right\|_{#2}}}
\newcommand{\pndot}[1]{\ensuremath{\left\|\cdot\right\|_{#1}}}
\newcommand{\Q}{\mathbb{Q}}
\newcommand{\R}{\mathbb{R}}
\newcommand{\set}[2]{\ensuremath{\left\{{#1}\;:\;\,{#2}\right\}}}
\newcommand{\st}{($*$)}
\newcommand{\tri}{{\displaystyle |\kern-.9pt|\kern-.9pt|}}
\newcommand{\tn}[1]{\ensuremath{\left|\kern-.9pt\left|\kern-.9pt\left|{#1}\right|\kern-.9pt\right|\kern-.9pt\right|}}
\newcommand{\tndot}{\ensuremath{\left|\kern-.9pt\left|\kern-.9pt\left|\cdot\right|\kern-.9pt\right|\kern-.9pt\right|}}
\newcommand{\ts}{\textstyle}

\DeclareMathOperator{\aspan}{span}
\DeclareMathOperator{\conv}{conv}
\DeclareMathOperator{\ext}{ext}

\begin{document}
\title{A topological characterization of dual strict convexity in Asplund spaces}
\begin{abstract}
Let $X$ be an Asplund space. We show that the existence of an equivalent norm on $X$ having a strictly convex dual norm is equivalent to the dual unit sphere $S_{X^*}$ (equivalently $X^*$) possessing a non-linear topological property called {\st}, which was introduced by J.~Orihuela, S.~Troyanski and the author.
\end{abstract}

\author{Richard J.\ Smith}
\address{School of Mathematical Sciences, University College Dublin, Belfield, Dublin 4, Ireland}
\email{richard.smith@maths.ucd.ie}
\thanks{The author wishes to thank J.~Orihuela, M.~Raja and S.~Troyanski for the stimulating conversations that helped to inspire this work, following a visit to the University of Murcia, Spain, in March 2019.}
\date{\today}
\maketitle

\section{Introduction}\label{intro}

All Banach spaces considered in this paper are real and all topological spaces are Hausdorff, with one explicitly stated exception. In \cite{lindenstrauss:76}, Lindenstrauss asked whether it is possible to characterize Banach spaces that admit an equivalent strictly convex norm. There has been a large collective effort on the part of a number of mathematicians to find such characterizations. The reader is referred to the Introduction of \cite{ost:12} for an account of this journey. In more recent years, concerning a number of questions in renorming theory in general, it has been recognised that effective characterizations can be given, at least in part, using the language of topology. We shall consider the following topological property.

\begin{defn}[{\cite[Definition 2.6]{ost:12}}]\label{defn_star}
We say that a topological space $X$ has {\st} if there exists a sequence $(\mathscr{U}_j)_{j<\omega}$ of families of open subsets of $X$, with the property that given any $x,y \in X$, there exists $j < \omega$ such that
\begin{enumerate}
\item $\{x,y\} \cap \bigcup\mathscr{U}_j$ is non-empty, and
\item $\{x,y\} \cap U$ is at most a singleton for all $U \in \mathscr{U}_j$.
\end{enumerate}
A sequence witnessing the {\st}-property on $X$ is called a \emph{{\st}-sequence}.
\end{defn}

The {\st} property is closely related to the so-called \emph{$G_\delta$-diagonal} property. A space $X$ is said to have a \emph{$G_\delta$-diagonal} if the diagonal is a $G_\delta$ set in $X^2$. Equivalently, $X$ has a $G_\delta$-diagonal if it possesses a \text{{\st}-sequence} in which every family is also a cover of $X$ (which renders property (1) above redundant). Such a sequence of covers is called a \emph{$G_\delta$-diagonal sequence}. The {\st} property also generalises the property of being a \emph{Gruenhage space}, which we won't define here. These properties have been studied in a purely topological context, as well as in relation to strictly convex norms on Banach spaces. We refer the reader to \cite{for:16,for:19,gru:84,ost:12,smith:09,smith:12} and the references therein for more details. 

Given a Banach space $X$, a subset $B\subseteq X$ and a norming subspace $F \subseteq X^*$, we say that $(B,\sigma(X,F))$ has \emph{{\st} with slices} if it is possible to find families $\mathscr{U}_j$ as in Definition \ref{defn_star}, satisfying the further condition that every element $U \in \bigcup_{j<\omega}\mathscr{U}_j$ is a $\sigma(X,F)$-open \emph{slice} of $B$, that is, $U$ is the intersection of $B$ with a $\sigma(X,F)$-open half-space of $X$. We have the following characterization of the existence of equivalent strictly convex norms on Banach spaces. We denote the unit sphere of $X$ by $S_X$.

\begin{thm}[{\cite[Theorem 2.7]{ost:12}}]\label{thm_star_slices} Let $X$ be a Banach space and let $F \subseteq X^*$ be a $1$-norming subspace. Then the following are equivalent.
\begin{enumerate}
\item $X$ admits an equivalent $\sigma(X,F)$-lower semicontinuous strictly convex norm;
\item $(X,\sigma(X,F))$ has {\st} with slices;
\item $(S_X,\sigma(X,F))$ has {\st} with slices.
\end{enumerate}
\end{thm}

Observe that Theorem \ref{thm_star_slices} (2) $\Rightarrow$ (3) is trivially true as it is easily seen that the {\st} property is hereditary, that is, it passes to subspaces. It should be recognised that Theorem \ref{thm_star_slices} is a linear-topological characterization of the existence of such norms, and is not fully non-linear. In the interests of obtaining a better and more easily verifiable characterization, it is natural to ask whether the additional slice condition in Theorem \ref{thm_star_slices} (2) and (3) is necessary. In general, it is indeed necessary:~while the existence of an equivalent $\sigma(X,F)$-lower semicontinuous strictly convex norm implies that $(X,\sigma(X,F))$ has {\st}, the converse is false -- see \cite[Theorem 2.7 and Example 1]{ost:12}.

On the other hand, if we consider dual spaces $X^*$, sometimes we are able to remove the additional slice condition. For the rest of the Introduction, we consider dual spaces $X^*$ only and fix $F=X\subseteq X^{**}$, so that $\sigma(X^*,F)$ is simply the usual $w^*$-topology on $X^*$.

\begin{thm}[{\cite[Theorem 3.1]{ost:12}}]\label{thm_scattered}
Let $K$ be a scattered compact space. Then the following are equivalent.
\begin{enumerate}
\item $C(K)^*$ admits an equivalent strictly convex dual norm;
\item $(C(K)^*,w^*)$ has {\st} with slices;
\item $(C(K)^*,w^*)$ has {\st};
\item $K$ has {\st}.
\end{enumerate}
\end{thm}

As above, (3) $\Rightarrow$ (4) is trivial, since {\st} passes to subspaces and $K$ embeds homeomorphically inside $(S_{C(K)^*},w^*)$, taken with respect to the canonical variation norm on $C(K)^*$. The most involved implication is (4) $\Rightarrow$ (1).

Of course, if we work in a dual space $X^*$ then we have the $w^*$-compactness of $B_{X^*}$, and the fact, by the Krein-Milman Theorem, that $B=\cl{\conv}^{w^*}(\ext(B))$ whenever $B\subseteq X^*$ is $w^*$-compact and convex (where $\ext(B)$ is the set of extreme points of $B$). The existence of extreme points can be used sometimes to pass from general $w^*$-open subsets of $B$ to $w^*$-open slices because, by Choquet's Lemma, any extreme point $e\in\ext(B)$ has a local base of $w^*$-open slices of $B$.

It turns out that in this note we require a slightly greater abundance of extreme points. Recall that a compact space $K$ is scattered precisely when $C(K)$ is an \emph{Asplund} space. The property of being an Asplund space has many equivalent formulations. We will require that fact that $X$ is Asplund if and only if $X^*$ has the \emph{Krein-Milman property}, that is, $B=\cl{\conv}^{\ndot}(\ext(B))$ whenever $B\subseteq X^*$ is norm-closed, convex and bounded.

The next theorem is our main result.

\begin{thm}\label{main_thm} Let $X$ be an Asplund space. Then the following are equivalent.
\begin{enumerate}
\item $X^*$ admits an equivalent strictly convex dual norm;
\item $(X^*,w^*)$ has {\st} with slices;
\item $(X^*,w^*)$ has {\st};
\item $(S_{X^*},w^*)$ has {\st}.
\end{enumerate}
\end{thm}

The implication (1) $\Rightarrow$ (2) is covered by Theorem \ref{thm_star_slices}, and (2) $\Rightarrow$ (3) and (3) $\Rightarrow$ (4) are trivial. Hence the only implication that requires work is (4) $\Rightarrow$ (1).

We conclude the Introduction by pointing out the following related result, which is stated in part.

\begin{thm}[{\cite[Theorem 1.6]{for:19}}]\label{Gdelta_thm}
Let $X$ be a Banach space. Then the following are equivalent.
\begin{enumerate}
\item $X^*$ admits an equivalent strictly convex dual norm;
\item $X^*$ admits an equivalent dual norm $\tndot$ such that $(S_{(X^*,\tndot)},w^*)$ has a $G_\delta$-diagonal.
\end{enumerate}
\end{thm}

It is worth comparing Theorems \ref{main_thm} and \ref{Gdelta_thm}. First, Theorem \ref{Gdelta_thm} applies to all Banach spaces. Second, Theorem \ref{Gdelta_thm} has no analogue of Theorem \ref{main_thm} (3). If $(X^*,w^*)$ has a $G_\delta$-diagonal then so does $(B_{X^*},w^*)$. This would force $X$ to be separable because compact spaces having $G_\delta$-diagonals are metrizable \cite[Theorem 2.13]{gru:84}. Third, in general, Theorem \ref{Gdelta_thm} (2) depends on the choice of equivalent dual norm. For example, let $K$ be the 1-point compactification of an uncountable discrete space. Then $C(K)^*$ admits an equivalent strictly convex dual norm e.g.~by Theorem \ref{thm_scattered} (though we stress that this fact was known long before the appearance of \cite{ost:12}). However, since $K$ is not metrizable, and embeds homeomorphically inside $(S_{C(K)^*},w^*)$, taken with respect to the canonical norm, it follows that the latter space does not have a $G_\delta$-diagonal. By contrast, owing to Theorem \ref{main_thm} (3), we can see that if Theorem \ref{main_thm} (4) holds with respect to one equivalent dual norm on $X^*$, then it will hold for all equivalent dual norms. 

The rest of the paper is organised as follows. The next section introduces a class of convex functions that is needed in the proof of Theorem \ref{main_thm} (4) $\Rightarrow$ (1), and relates these functions to the notion of $F$-distance, which was used in \cite{ost:12,ot:08}. The proof of Theorem \ref{main_thm} (4) $\Rightarrow$ (1) is given in Section \ref{sect_proof}.

\section{The functions $\varphi$ and $F$-distance}\label{sect_F-distance}

Let $X$ be a Banach space with norm $\ndot$, and $F \subseteq X^*$ a subspace. Let $\Gamma \subseteq S_F \times \R$ have the property that
\begin{equation}\label{eqn_21}
M \;:=\; \sup\set{-a}{(f,a) \in \Gamma} \;<\; \infty.
\end{equation}
Define a function $\map{\varphi}{X}{[0,\infty)}$ by
\[
\varphi(x) \;=\; \sup\big(\set{f(x)-a}{(f,a)\in\Gamma} \cup \{0\}\big).
\]
We can see that $\varphi(x)\leqslant \n{x}+M$ for all $x \in X$. It is clear that $\varphi$ is convex, $1$-Lipschitz and $\sigma(X,F)$-lower semicontinuous. As we will see below, the function $\varphi$ is closely related to the notion of {\em $F$-distance}, which was introduced in \cite[Proposition 1]{ot:08}. The next lemma is very similar to \cite[Proposition 2.2]{ost:12}. We include it because its conclusion is slightly more specific and its proof has been much simplified. 

\begin{lem}\label{convexity_lem}
Suppose that
\[
\lambda \;:=\; \varphi(x) \;=\; \varphi(y) \;=\; \varphi({\ts\frac{1}{2}(x+y)}) \;>\; 0.
\]
Then there exists a sequence $(f_k,a_k) \subseteq \Gamma$ such that
\[
 f_k(x) - a_k ,\, f_k(y)-a_k \;\to\; \lambda,
\]
as $k\to\infty$.
\end{lem}

\begin{proof}
The statement above is a simple consequence of the following easy fact. Define $\lambda=\max\{\varphi(x),\varphi(y)\}$ and suppose $\varphi(\frac{1}{2}(x+y))>r$. Then there exists $(f,a)\in\Gamma$ such that $f(\frac{1}{2}(x+y))-a>r$. Since $f(y)-a \leqslant \varphi(y)\leqslant \lambda$, we have
\[
f(x)-a \;=\; f(x+y)-2a - (f(y)-a) \;>\; 2r-\lambda.
\]
Similarly, $f(y)-a > 2r-\lambda$.
\end{proof}

We take a brief detour to explore the relationship between the functions $\varphi$ and so-called $F$-distance. We feel that this detour is justified, given the clear connections between this paper and \cite{ost:12,ot:08}, wherein $F$-distance is introduced and put to use. We define a seminorm on $X^{**}$ by
\[
\pn{\xi}{F} \;=\; \sup\set{\xi(f)}{f \in S_F}.
\]
This seminorm is easily seen to be $\sigma(X^{**},F)$-lower semicontinuous (observe that $\sigma(X^{**},F)$ is non-Hausdorff in general). Then, given a non-empty convex bounded set $A \subseteq X$, define the {\em $F$-distance} from $x \in X$ to $A$ by
\[
\psi(x) \;=\; \inf\set{\pn{x-\xi}{F}}{\xi \in \cl{A}^{w^*}},
\]
where $\cl{A}^{w^*}$ denotes closure of $A$ with respect to the $w^*$-topology of $X^{**}$.

This definition is arguably difficult to penetrate. We spend the remainder of this section showing that it can be recast in a more transparent way by considering the functions $\varphi$. Specifically, we will define a set $\Gamma$ as above in such a way that the associated function $\varphi$ is equal to $\psi$. Given $(f,a) \in S_F \times \R$, define the $\sigma(X,F)$-open halfspace
\[
H_{f,a} \;=\; \set{x \in X}{f(x)>a}.
\]
Then define
\[
\Gamma \;=\; \set{(f,a) \in S_F\times\R}{A \cap H_{f,a} = \varnothing}.
\]
Since there exists $x_0 \in A$, we have $-\n{x_0}\leqslant f(x_0) \leqslant a$ for all $(f,a)\in\Gamma$, which means $M$, as defined in \eqref{eqn_21}, is no greater than $\n{x_0}$, and hence $\varphi$ is well-defined.

Now let $x \in X$ and $(f,a) \in \Gamma$. Since $f(y)\leqslant a$ for all $y \in A$, we have $\xi(f)\leqslant a$ for all $\xi \in \cl{A}^{w^*}$. This means that $f(x)-a \leqslant f(x)-\xi(f) \leqslant \pn{x-\xi}{F}$, giving $\varphi(x)\leqslant \psi(x)$. To see the reverse inequality, we can assume that $\psi(x)>0$. Given $\lambda \in (0,\psi(x))$, define the $\sigma(X^{**},F)$-closed set
\[
B \;=\; \set{\eta \in X^{**}}{\pn{x-\eta}{F}\leqslant\lambda}.
\]
It is clear that $\cl{A}^{w^*} \cap B$ is empty. Since $\cl{A}^{w^*}$ is $w^*$-compact and the $w^*$-topology on $X^{**}$ is finer than $\sigma(X^{**},F)$, it follows that $B-\cl{A}^{w^*}$ is $\sigma(X^{**},F)$-closed. By the Hahn-Banach separation theorem, there exists $f \in S_F$ such that $\eta(f) > 0$ for all $\eta \in B-\cl{A}^{w^*}$. If we set
\[
a \;=\; \max\set{\xi(f)}{\xi \in \cl{A}^{w^*}},
\]
then $\eta(f)> a$ whenever $\eta \in B$. Moreover $H_{f,a} \cap A$ is empty and $(f,a) \in \Gamma$. Let $\xi \in S_{X^{**}}$ satisfy $\xi(f)=1$. Then $x - \lambda \xi \in B$, giving $f(x)-\lambda > a$ and $\varphi(x)\geqslant f(x)-a > \lambda$. Since this holds for all such $\lambda$, it follows that $\varphi(x) \geqslant \psi(x)$ as required. This concludes our discussion of $F$-distance.

\section{The proof of Theorem \ref{main_thm} (4) $\Rightarrow$ (1)}\label{sect_proof}

Hereafter, we will consider the dual space $X^*$ of $X$ itself, and set $F=X \subseteq X^{**}$, so that $\sigma(X^*,F)$ is simply the $w^*$-topology on $X^*$, the set $\Gamma$ above will be a subset of $S_X \times \R$, and $\map{\varphi}{X^*}{[0,\infty)}$ is given by
\[
 \varphi(f) \;=\; \sup\big(\set{f(x)-a}{(x,a) \in \Gamma} \cup \{0\}\big).
\]
Given $(x,a) \in S_X \times \R$, we define the $w^*$-open halfspace
\[
H_{x,a} \;=\; \set{f \in X^*}{f(x)>a}.
\]

Now we can present the main proof. It uses a set-theoretic derivation process that expands on the one used in the proof of Theorem \ref{Gdelta_thm}.

\begin{proof}[Proof of Theorem \ref{main_thm} (4) $\Rightarrow$ (1)]
We will prove the contrapositive. Let $\ndot$ denote the original dual norm on $X^*$. Let $(\mathscr{V}_j)$ be a sequence of families of $w^*$-open subsets of $S_{X^*}$. We will define a dual norm $\tndot$ on $X^*$ having the property that if $\tndot$ is not strictly convex, then $(\mathscr{V}_j)$ is not a {\st}-sequence for $(S_{X^*},w^*)$. 

Define
\[
 \mathscr{U}_j \;=\; \set{U\subseteq X^*}{U \text{ is $w^*$-open and }U \cap S_{X^*} \in \mathscr{V}_j}.
\]

To define $\tndot$, we will need to consider a derivation process on $B_{X^*}$ that is indexed by $(\omega \times \Q)^{<\omega}$, the set of finite sequences of elements of $\omega \times \Q$. For each $s \in (\omega \times \Q)^{<\omega}$, we will define a $w^*$-compact convex subset $B_s$ of $B_{X^*}$. Set $B_\varnothing = B_{X^*}$. Let $s \in (\omega\times\Q)^{<\omega}$ and suppose that $B_s$ has been defined. Given $j \in \omega$, set
\[
\Gamma_{s,j} \;=\; \set{(x,a) \in S_X \times (0,1)}{H_{x,a} \cap B_s \subseteq U\text{ for some }U \in \mathscr{U}_j}.
\]
Given $q\in\Q$, define the family of $w^*$-open halfspaces
\[
 \mathscr{S}_{s,j,q} \;=\; \set{H_{x,a+q}}{(x,a) \in \Gamma_{s,j}},
\]
and set
\begin{equation}\label{eqn_6}
 B_{s^\frown(j,q)} \;=\; B_s \setminus \bigcup\mathscr{S}_{s,j,q}.
\end{equation}

Given $s \in (\omega\times\Q)^{<\omega}$ and $j<\omega$, define $\map{\varphi_{s,j}}{X^*}{[0,\infty)}$ by 
\[
\varphi_{s,j}(f) \;=\; \sup\big(\set{f(x)-a}{(x,a) \in \Gamma_{s,j}} \cup \{0\}\big).
\]
All such maps are convex, $1$-Lipschitz and $w^*$-lower semicontinuous. The following observation will be of use later on. If $s \preccurlyeq t$ then $B_t\subseteq B_s$, meaning that $\Gamma_{s,j}\subseteq \Gamma_{t,j}$, and so $\varphi_{s,j}\leqslant \varphi_{t,j}$. 

Define
\[
 C_{s,j,r} \;=\; \set{f \in X^*}{\varphi_{s,j}(f)^2 + \varphi_{s,j}(-f)^2 \leqslant r},
\]
whenever $r\in\Q$ is positive. All such sets are symmetric, convex, $w^*$-compact and contain a norm-open neighbourhood of the origin. Let $\pndot{s,j,r}$ denote its corresponding $w^*$-lower semicontinuous Minkowski functional. Now define $\tndot$ by setting
\[
 \tn{f}^2 \;=\; \n{f}^2 + \sum_{s,j,r} c_{s,j,r}\pn{f}{s,j,r}^2,
\]
where the constants $c_{s,j,r}>0$ are chosen in such a way that the sum converges uniformly on bounded sets. Being $w^*$-lower semicontinuous, $\tndot$ is a dual norm.

Let us assume that $\tndot$ is not strictly convex. Then we can find distinct $f,g \in X^*$ such that
\[
 \tn{f} \;=\; \tn{g} \;=\; \tn{{\ts\frac{1}{2}(f+g)}},
\]
By standard convexity arguments (see e.g.~\cite[Fact II.2.3]{dgz:93}), we have
\[
 \n{f} \;=\; \n{g} \;=\; \n{{\ts\frac{1}{2}(f+g)}}.
\]
Therefore, by rescaling if necessary, we shall assume that $f,g,\frac{1}{2}(f+g) \in S_{X^*}$. Again, by convexity arguments
\[
 \pn{f}{s,j,r} \;=\; \pn{g}{s,j,r} \;=\; \pn{{\ts\frac{1}{2}(f+g)}}{s,j,r} 
\]
for all $s,j,r$ as above, and a final application of convexity arguments yields
\begin{equation}\label{eqn_1}
\varphi_{s,j}(f) \;=\; \varphi_{s,j}(g) \;=\; \varphi_{s,j}({\ts\frac{1}{2}(f+g)}),
\end{equation}
for all $s \in (\omega \times \Q)^{<\omega}$ and $j<\omega$.

Let $\lambda_{s,j}$ denote the common value in \eqref{eqn_1} above. We will build a sequence $s_0 \prec s_1 \prec s_2 \prec \dots$ such that $f,g \in B_{s_n}$ for all $n < \omega$. Set $s_0=\varnothing$. Given $s_n$, we build $s_{n+1}$ using $n+1$ intermediate stages
\[
s_n \;=:\; s_{n,0} \prec s_{n,1} \prec \dots \prec s_{n,n+1} \;=:\; s_{n+1}.
\]
Suppose that $s_{n,j}$, $j \leqslant n$, has been constructed. Take $q \in [\lambda_{s_{n,j},j},(1+2^{-n})\lambda_{s_{n,j},j}] \cap \Q$ and set $s_{n,j+1}={s_{n,j}}^\frown(j,q)$.  Since $\lambda_{s_{n,j},j} \leqslant q$, for all $(x,a) \in \Gamma_{s_{n,j},j}$ we have $f(x)-a \leqslant \lambda_{s_{n,j},j} \leqslant q$ and so $f \notin H_{x,a+q}$. Likewise $g \notin H_{x,a+q}$. Therefore
\[
f,g \in B_{s_{n,j}}\setminus\bigcup\mathscr{S}_{s_{n,j},j,q} \;=\; B_{{{s_{n,j}}^\frown(j,q)}} \;=\; B_{s_{n,j+1}}.
\]
In this way we build the sequence $(s_n)$.

Set $B=\bigcap_{n<\omega} B_{s_n}$. Notice that, for all $j<\omega$, $(\lambda_{s_n,j})$ is an increasing sequence bounded above by $1$. Given $j<\omega$, define $\lambda_j=\lim_{n\to\infty} \lambda_{s_n,j}$ and set
\[
 \Gamma_j \;=\; \set{(x,a) \in S_X \times (0,1)}{H_{x,a} \cap B \subseteq U\text{ for some }U \in \mathscr{U}_j}.
\]
Since $B \subseteq B_{s_n}$ for all $n<\omega$, we have $\Gamma_{s_n,j} \subseteq \Gamma_j$ for all $n,j<\omega$. Define also
\[
 \varphi_j(h) \;=\; \sup\big(\set{h(x)-a}{(x,a) \in \Gamma_j} \cup \{0\}\big),
\]
and 
\[
 \mu_j \;=\; \sup_{h \in B} \varphi_j(h).
\]

We proceed to make three claims. Our first claim is that
\begin{equation}\label{eqn_10}
\varphi_j(h) \;=\; \lim_{n \to \infty} \varphi_{s_n,j}(h), 
\end{equation}
for all $h \in B$. Indeed, $\varphi_{s_n,j}(h)\leqslant \varphi_j(h)$ because $\Gamma_{s_n,j} \subseteq \Gamma_j$. Now let $h \in B$ and $\ep>0$. Take $(x,a) \in \Gamma_j$ such that $h(x)-a > \varphi_j(h) - \ep$. Pick $U \in \mathscr{U}_j$ such that $H_{x,a} \cap B \subseteq U$. It follows that
\[
\cl{H_{x,a+\ep}}^{w^*} \cap B \;\subseteq\; H_{x,a} \cap B \;\subseteq\; U,
\]
so by a $w^*$-compactness argument there exists $n<\omega$ such that
\[
 \cl{H_{x,a+\ep}}^{w^*} \cap B_{s_n} \;\subseteq\; U.
\]
Hence $(x,a+\ep) \in \Gamma_{s_n,j}$, giving 
\[
\varphi_{s_n,j}(h) \;\geqslant\; h(x)-a-\ep \;>\; \varphi_j(h) - 2\ep. 
\]
This completes the proof of the claim.

Our second claim is that $\mu_j=\lambda_j$ for all $j<\omega$. First, by applying \eqref{eqn_10} we get
\begin{equation}\label{eqn_11}
\lambda_j \;=\; \varphi_j(f)\;=\; \varphi_j(g) \;=\; \varphi_j({\ts\frac{1}{2}(f+g)}) \;\leqslant\; \mu_j. 
\end{equation}
Now let $h \in B$. Given $n \geqslant j$, by \eqref{eqn_6} and the definition of $s_{n,j+1}$ we have
\[
 h \in B \;\subseteq\; B_{s_{n,j+1}} \;=\; B_{s_{n,j}}\setminus\bigcup\mathscr{S}_{s_{n,j},j,q},
\]
where $q \leqslant (1+2^{-n})\lambda_{s_{n,j},j}$. Therefore,
\[
 h(x) \;\leqslant\; a + q \;\leqslant\; a+(1+2^{-n})\lambda_{s_{n,j},j},
\]
whenever $(x,a) \in \Gamma_{s_{n,j},j}$. It follows that 
\[
\varphi_{s_n,j}(h)\;\leqslant\;\varphi_{s_{n,j},j}(h) \;\leqslant\; (1+2^{-n})\lambda_{s_{n,j},j} \;\leqslant\; (1+2^{-n})\lambda_j,
\]
whenever $n \geqslant j$. Consequently,
\[
 \varphi_j(h) \;=\; \lim_{n\to\infty} \varphi_{s_n,j}(h) \;\leqslant\; \lambda_j.
\]
Since this holds for all $h \in B$, we have $\mu_j \leqslant \lambda_j$. Coupled with \eqref{eqn_11}, this completes the proof of the second claim.

Our third claim is that, given $j<\omega$, if $\lambda_j=0$ then $\ext(B) \cap \bigcup \mathscr{U}_j = \varnothing$. Indeed, if $h \in \ext(B) \cap \bigcup\mathscr{U}_j$ then there exists $U \in \mathscr{U}_j$ such that $h \in U$ and, by Choquet's Lemma, there exists $(x,a) \in S_X \cap (0,1)$ such that
\[
 h \in H_{x,a} \cap B \;\subseteq\; U.
\]
Thus $(x,a) \in \Gamma_j$. From this we deduce that
\[
\lambda_j \;=\; \mu_j \;\geqslant\; \varphi_j(h) \;\geqslant\; h(x)-a \;>\; 0,
\]
which yields the third claim.

Now suppose that $\lambda_j>0$. According to Lemma \ref{convexity_lem}, there is a sequence of pairs $(x_{k,j},a_{k,j}) \subseteq \Gamma_j$ such that
\begin{equation}\label{eqn_30}
 f(x_{k,j})-a_{k,j},\, g(x_{k,j})-a_{k,j} \;\to\; \lambda_j,
\end{equation}
as $k\to\infty$. Without loss of generality we can assume that the sequences of real numbers $(f(x_{k,j}))$, $(g(x_{k,j}))$ and $(a_{k,j})$ converge. Set $a_j = \lim_{k\to\infty} a_{k,j}$ and let $\xi_j \in B_{X^{**}}$ be an accumulation point of the sequence $(x_{k,j})$ in the $w^*$-topology of $X^{**}$. From \eqref{eqn_30} we have 
\begin{equation}\label{eqn_31}
\xi_j(f) \;=\; \xi_j(g) \;=\; \lambda_j + a_j.
\end{equation}
Next, we claim that
\begin{equation}\label{eqn_32}
\xi_j(h) \leqslant \lambda_j + a_j, 
\end{equation}
for all $h \in B$. Indeed, let $h \in B$ and suppose that $\xi_j(h) \geqslant \lambda_j+a_j$. Extract a subsequence $(x_{k_i,j},a_{k_i,j})$ of $(x_{k,j},a_{k,j})$ such that $h(x_{k_i,j}) \to \xi_j(h)$ as $i \to \infty$. Then 
\[
 \lambda_j \;=\; \mu_j \;\geqslant\; \varphi_j(h) \;\geqslant\; h(x_{k_i,j})-a_{k_i,j} \;\to\; \xi_j(h)-a_j \;\geqslant\; \lambda_j,
\]
as $i\to\infty$. This forces $\xi_j(h)=\lambda_j+a_j$. Furthermore, observe that we have also shown that
\begin{equation}\label{eqn_33}
\phi_j(h) \;=\; \lambda_j,
\end{equation}
whenever $h \in B$ and $\xi_j(h)=\lambda_j+a_j$.

We can find such $\xi_j$ and $a_j$ whenever $\lambda_j>0$. Let $\xi \in S_{X^{**}}$ satisfy $\xi(\frac{1}{2}(f+g))=1$ and define
\[
 D \;=\; \set{h \in B}{\xi(h)=1 \text{ and }\xi_j(h)=\lambda_j+a_j \text{ for all }j<\omega\text{ satisfying }\lambda_j>0}.
\]
By \eqref{eqn_31}, we have $f,g \in D$. Using \eqref{eqn_32} and the fact that $D\subseteq S_{X^*}$, we can deduce that $\ext(D)\subseteq \ext(B) \cap S_{X^*}$. Since $X$ is an Asplund space and $D$ is a norm-closed bounded convex set in $X^*$, we have $D=\cl{\aspan}^{\ndot}(\ext(D))$. In particular, since $f$ and $g$ are distinct, we can find distinct elements $d,e \in \ext(D)\subseteq\ext(B)\cap S_{X^*}$. Let $j<\omega$. Suppose that $\lambda_j=0$. Then from our third claim above we have $d,e \notin \bigcup\mathscr{U}_j$. Instead, if $\lambda_j>0$, then by \eqref{eqn_33} and Lemma \ref{convexity_lem}, there exists $(x,a) \in \Gamma_j$ such that 
\[
d(x)-a,\, e(x)-a \;>\; {\ts \frac{1}{2}}\lambda_j \;>\; 0,
\]
meaning that $d,e \in H_{x,a}$. According to the definition of $\Gamma_j$, it follows that there exists $U \in \mathscr{U}_j$ satisfying $d,e \in U$.

In summary, for all $j<\omega$, either $d,e \notin \bigcup\mathscr{U}_j$ or there exists $U \in \mathscr{U}_j$ such that $d,e \in U$. Since $d,e \in S_{X^*}$, the same observation applies to the families $\mathscr{V}_j$. It follows that $(\mathscr{V}_j)$ is not a {\st}-sequence for $S_{X^*}$. This completes the proof.
\end{proof}

We finish with a couple of open questions. We do not know if the assumption that $X$ is Asplund can be removed from Theorem \ref{main_thm}. Furthermore, we do not know if the assumption that $K$ is scattered can be removed from Theorem \ref{thm_scattered}. 

\end{document}